\documentclass[12pt]{amsart}

\usepackage[pdftex,colorlinks,citecolor=blue,urlcolor=blue]{hyperref}
\usepackage{amssymb}
\usepackage{graphicx,color}
\usepackage{amsmath}
\usepackage{comment}
\usepackage{MnSymbol}
\usepackage{cleveref}

\newtheorem{theorem}{Theorem}[section]
\newtheorem{proposition}[theorem]{Proposition}
\newtheorem{lemma}[theorem]{Lemma}
\newtheorem{corollary}[theorem]{Corollary}
\theoremstyle{definition}
\newtheorem{definition}[theorem]{Definition}

\newtheorem{question}[theorem]{Question}

\newtheorem{remark}[theorem]{Remark}

\DeclareMathOperator{\lk}{\mathrm{lk}}

\title{On flag spheres with few equators}
\author[L. Venturello]{Lorenzo Venturello}
\email{lven@kth.se, lorenzo.venturello@hotmail.it}
\address{Department of Mathematics, KTH Royal Institute of Technology, Stockholm, Sweden}

\date{\today}

\begin{document}

\maketitle
\begin{abstract}
    In this note we construct a flag simplicial $3$-sphere $\Delta$ with the following properties:
    \begin{itemize}
        \item[-] $\Delta$ is not a suspension;
        \item[-] $\Delta$ has no edge that can be contracted to obtain another flag sphere;
        \item[-] The only equators (induced subcomplexes which are spheres of codimension $1$) of $\Delta$ are vertex links.
    \end{itemize}
    Our construction has $12$ vertices, the minimum number of vertices such a simplicial complex can have. This answers a question posed by Chudnovsky and Nevo.
\end{abstract}
\section{Introduction}
This manuscript is concerned with the construction of a particular \emph{flag} simplicial sphere. A simplicial sphere, more generally a simplicial complex, is flag if it is equal to the clique complex of its graph. The interest in this family sparks from many directions, most notably in the study of non-positively curved cubical manifold. Indeed, the work of Gromov established that a simply connected cubical complex is CAT(0) if and only if the vertex links are flag simplicial complexes. If the cubical complex is also a manifold, then the property of being CAT(0) is equivalent to having flag simplicial spheres as vertex links. Moreover, via the Stanley-Reisner correspondence, flag simplicial spheres correspond to quadratic squarefree monomial ideals, whose quotient is a Gorenstein ring.\\
One of the main problems in the subject consists of characterizing linear inequalities satisfied by the number of faces in each dimension of any flag sphere. In particular, a conjecture of Gal \cite{Gal} predicts $\lfloor \frac{d+1}{2}\rfloor$ many particular linear inequalities among the face numbers of a flag simplicial $d$-sphere. It is easily seen that they are not valid for all spheres, so being flag is really essential. These inequalities can be presented as nonnegativity conditions on an integer linear transformation of the face vector of a flag simplicial sphere, called the $\gamma$-vector. Since its introduction a lot of research has been carried out to prove nonnegativity of the $\gamma$-vector of special families of flag spheres, and few strengthenings have been proposed (see e.g., \cite{NP}). \\
In order to address Gal's conjecture one could look for a reduction-type strategy as proposed in \cite{CN}, and apply induction (on the number of vertices and on the dimension). More precisely, we investigate if every flag sphere can be ``reduced" to a smaller flag sphere by changing the $\gamma$-vector in a controlled way. We consider three ways to perform this reduction: one can either contract an edge while preserving flagness (if such an edge exists), remove suspension points (if the sphere is a suspension), or dissect the sphere by cutting along a codimension $1$ sphere and then cone over the boundary of the two balls obtained in this way. The last operation can be always performed, since we can cut a sphere along the link of any vertex. However, in this case one of the two new spheres would be isomorphic to the one we started from, so the operation did not decrease the number of vertices. Hence, in order to have an interesting dissection, we need an \emph{induced equator}, i.e., an induced subcomplex which is a codimension $1$ sphere, that is not the link of any vertex. Recently, Chudnowsky and Nevo \cite{CN} studied such subcomplexes and proposed a structural problem, which would imply a negative answer to the following question. 
\begin{question}\cite[Problem 1.5]{CN}\label{qu: question}
    Does there exist a flag simplicial $d$-sphere $\Delta$ such that:
    \begin{itemize}
        \item[(i)] $\Delta$ is not a suspension,
        \item[(ii)] $\Delta$ does not have any edge which can be contracted to obtain another flag sphere,
        \item[(iii)] All the induced equators of $\Delta$ are vertex links?
    \end{itemize}
\end{question}
In other words, \Cref{qu: question} asks if there exists spheres which are ``irreducible" with respect to the reduction steps described above. The precise statement of \cite[Problem 1.5]{CN} is actually stronger than a negative answer to \Cref{qu: question}, as the authors require the existence for every vertex of an equator which is not a vertex link and it does not contain the fixed vertex. In \cite{CN} the authors answered this question in the negative for $d=2$, by showing that every flag $2$-sphere which is not the octahedron (which is a suspension) has a contractible edge. Moreover, they conjecture that $\gamma$-vectors of induced equators give lower bounds for the $\gamma$-vector of the entire flag sphere, and that this fact would imply Gal's conjecture. Perhaps surprisingly, we will show that \Cref{qu: question} has a positive answer already in dimension $3$.\\ 
Our contribution consists of constructing a flag simplicial $3$-sphere $\Delta_{12,33}$ with the three properties listed in \Cref{qu: question}, by glueing two particular simplicial solid tori along their boundary. The flag $3$-sphere obtained has $12$ vertices, which is the smallest number of vertices a sphere with these properties can have (see \Cref{rem: delta is minimal}), $33$ facets, and its $\gamma$-vector is $\gamma(\Delta_{12,33})=(1,4,1)$. By taking the $k$-fold simplicial join of $\Delta_{12,33}$ we obtain $(4k+3)$-spheres which also give a positive answer to \Cref{qu: question}.

\section{Notation and preliminaries}
An \emph{abstract simplicial complex} $\Delta$ on a finite set $V$ is collection of subsets of $V$ closed under inclusion. The simplicial complex $2^V$ is a $(|V|-1)$-simplex. An element $F\in\Delta$ is a $(|F|-1)$-dimensional \emph{face}, and the \emph{dimension} of $\Delta$ is the maximum of the dimension of one of its faces. Singletons of $\Delta$ are called \emph{vertices}. A simplicial complex is determined by its \emph{facets}, i.e., faces which are maximal w.r.t. inclusion, and we denote by $\langle F_1,\dots,F_k\rangle$ the simplicial complex with facets $F_1,\dots,F_k$. Every subset $W\subset V$ defines a subcomplex $\Delta[W]:=\{F\in \Delta: F\subset W\}$ of $\Delta$. Subcomplexes of this form are \emph{induced}. For every $F\in\Delta$, the \emph{link} of $F$ in $\Delta$ is the simplicial complex $\lk_{\Delta}(F):=\{G\in\Delta: F\cup G\in\Delta, F\cap G=\emptyset\}$. \\
We can associate to every abstract simplicial complex its \emph{geometric realization} $|\Delta|$, i.e., the topological space triangulated by $\Delta$. If $|\Delta|\cong\mathbb{S}^{d}$ then $\Delta$ is a \emph{simplicial $d$-sphere}. We now define an operation on simplicial complexes. the \emph{simplicial join} of two simplicial complexes $\Delta$ and $\Gamma$ is the simplicial complex 
\[
    \Delta*\Gamma = \{F\cup G: F\in\Delta, G\in\Gamma\}.
\]
If $\Gamma=\langle \{v\} \rangle$ then $\{v\}*\Delta:=\Delta*\Gamma$ is the \emph{cone} over $\Delta$. If $\Gamma=\langle \{v\},\{w\}\rangle$ then $\Delta*\Gamma$ is the \emph{suspension} over $\Delta$. We say that a simplicial complex is isomorphic to $\Diamond_d$ is it is isomorphic to the $d$-fold suspension of the $0$-sphere. A simplicial complex isomorphic to $\Diamond_d$ is then a simplicial $(d-1)$-sphere on $2d$ vertices which can be realized geometrically as the boundary complex of the $d$-dimensional \emph{cross-polytope}. In particular, $\Diamond_2$ is a $4$-cycle, and $\Diamond_3$ is isomorphic to the boundary complex of an octahedron. If $\Delta\cong\Diamond_d$, then every vertex of $\Delta$ as a unique non-adjacent vertex in $\Delta$, and $\Delta$ consists precisely of all possible subsets of the vertex set which do not contain a pair of non-adjacent vertices. We now fix a further piece of notation. Consider two ordered tuples $[x_1,\dots,x_d]$ and $[y_1,\dots,y_d]$ of $2d$ distinct elements. We define
    \[
        \Diamond([x_1,\dots,x_d],[y_1,\dots,y_d]):=\langle \{x_1\},\{y_1\}\rangle*\cdots*\langle \{x_d\}, \{y_d\}\rangle.
    \] 
Evidently $\Diamond([x_1,\dots,x_d],[y_1,\dots,y_d])\cong \Diamond_d$, and $\{x_i,y_i\}\neq \Diamond([x_1,\dots,x_d],[y_1,\dots,y_d])$ for every $i=1,\dots,d$.\\
We now define the main class of simplicial complexes we treated in this article.  
\begin{definition}
    A simplicial complex $\Delta$ on $V$ is \emph{flag} if every non-face (a subset $F\subset V$ with $F\notin\Delta$) which does not properly contain any other non-face has cardinality $2$.
\end{definition}
Equivalently, a simplicial complex is flag if it is the clique complex of its graph. The simplicial sphere $\Diamond([x_1,\dots,x_d],[y_1,\dots,y_d])$ is flag, as the subsets $\{x_1,y_i\}$ are the minimal non-faces. In fact, it is the unique vertex minimal flag simplicial $(d-1)$-sphere. Moreover, any flag simplicial $3$-sphere can be obtained from $\Diamond_4$ via a sequence of edge subdivisions and edge contractions. We refer to \cite{LuN} for the details and the more general result for flag PL spheres of any dimension. However, not all edges can be contracted preserving the flag property. We say that an edge of $\Delta$ is \emph{contractible} if it is not contained in any induced $4$-cycle of $\Delta$, i.e., an induced subcomplex isomorphic to a $4$-cycle.

\section{The construction}\label{sec 2}
In this section we will construct a simplicial complex by identifying two triangulated solid tori along their isomorphic boundary. The resulting complex is a simplicial $3$-manifold and all manifolds which can be obtained in this way are called \emph{lens spaces}. We will argue in \Cref{thm: 3-sphere} that our construction is indeed a flag simplicial $3$-sphere. 
%
We define
\begin{align*}
\Gamma_1:=&\{v_1\}*\Diamond([x_1,x_2,x_3],[y_1,y_2,y_3])\cup \nonumber\\ &\{v_2\}*\Diamond([y_1,y_2,y_3],[z_1,z_2,z_3])\cup\\ &\{v_3\}*\Diamond([x_1,x_2,x_3],[z_3,z_1,z_2]).\nonumber
\end{align*}
$\Gamma_1$ is the union of $3$ a simplicial $3$-balls, each isomorphic yo the cone over $\Diamond_3$. The three balls intersect pairwise in the $2$-simplices with vertices $\{x_1,x_2,x_3\}$, $\{y_1,y_2,y_3\}$ and $\{z_1,z_2,z_3\}$, in a way that $|\Gamma_1|\cong \mathbb{D}^2\times \mathbb{S}^1$. The boundary complex $\partial\Gamma_1$ depicted in \Cref{fig: link}, which triangulates the torus, has the following facets:
\begin{align}
   \partial\Gamma_1=\langle &x_1x_2y_3, x_1y_2x_3, x_2x_3y_1, x_1y_2y_3, x_2y_1y_3, x_3y_1y_2,\nonumber\\\label{eq: boundary}
    &y_1y_2z_3, y_1y_3z_2, y_2y_3z_1, y_1z_2z_3, y_2z_1z_3, y_3z_1z_2,\\
   &x_1x_2z_2, x_1x_3z_1, x_2x_3z_3, x_1z_1z_2, x_2z_2z_3, x_3z_1z_3\rangle.\nonumber
\end{align}
Here and in the rest of this article we will shorten the notation by identifying a face with a string of vertices. For instance $x_1x_2y_3$ denotes the set $\{x_1,x_2,y_3\}$.\\

We will now construct a second simplicial complex $\Gamma_2$, such that $|\Gamma_2|\cong \mathbb{D}^2\times \mathbb{S}^1$ and $\Gamma_1\cap\Gamma_2=\partial\Gamma_1=\partial\Gamma_2$. To do so, we consider a $2$-simplex $\Delta_2=\langle\{a,b,c\}\rangle$ and a $3$-cycle $C_3=\langle\{\{1,2\},\{2,3\},\{1,3\}\rangle$. We can triangulate each of the prisms $\Delta_2\times \langle\{1,2\}\rangle$, $\Delta_2\times \langle\{2,3\}\rangle$ and $\Delta_2\times \langle\{1,3\}\rangle$ individually. As these prisms intersect pairwise in a $2$-simplex, their union is a simplicial refinement of $\Delta_2\times C_3$. By \cite[Section 6.2.2]{DRS}, a triangulation of $\Delta_2\times \Delta_1$ is given in the following way: after ordering the vertices of $\Delta_2$ and $\Delta_1$ we can represent every pair of vertices as an integer point in the square $[1,3]\times [1,2]$. Facet of the corresponding triangulation are given by the $3$ non-decreasing paths from $(1,1)$ to $(3,2)$. We order the vertices of $\Delta_2$ as $a<b<c$, and the vertices of the three edges in the following way: $1<2$, $2<3$ and $3<1$. Clearly, this is not a total order on the vertices of $C_3$, but only on each of the edges. 
We then obtain $9$ facets, as depicted in \Cref{fig: staircases}.
\begin{figure}[h]
    \centering
    \includegraphics[scale=0.77]{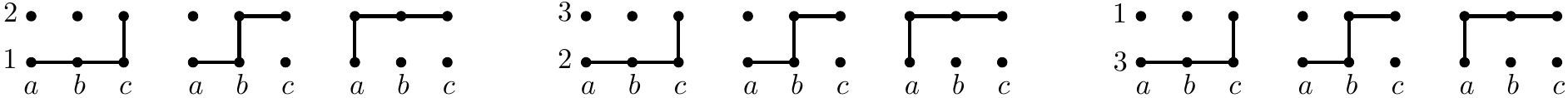}
    \caption{The $9$ facets of $\Gamma_2$}
    \label{fig: staircases}
\end{figure}
More explicitly, we list these facets representing coordinates $(i,j)$ as $i_j$.
\begin{align}\label{eq: facets abc}
    F:=\{ \nonumber&\underline{a_1b_1c_1}c_2,\underline{a_1b_1}b_2\underline{c_2},\underline{a_1}a_2\underline{b_2c_2},\\ 
    &\underline{a_2b_2c_2}c_3, \underline{a_2b_2}b_3\underline{c_3},\underline{a_2}a_3\underline{b_3c_3},\\
    \nonumber&\underline{a_3b_3c_3}c_1, \underline{a_3b_3}b_1\underline{c_1}, \underline{a_3}a_1\underline{b_1c_1}\}.
\end{align}
Consider now the following bijection: 
\begin{align}
    \varphi:~ a_1 \mapsto z_2, a_2 \mapsto z_3, a_3 \mapsto z_1,
    b_1 \mapsto y_3, b_2 \mapsto y_1, b_3 \mapsto y_2,
    c_1 \mapsto x_1, c_2 \mapsto x_2, c_3 \mapsto x_3. 
\end{align} 
The simplicial complex $\Gamma_2:=\langle \varphi(F)\rangle$ is homeomorphic to $|\Delta_2\times C_3|\cong \mathbb{D}^2\times \mathbb{S}^1$ and it is supported on a subset of the vertex set of $\Gamma_1$. As every face in the interior of $\Gamma_1$ contains one of the vertices in $\{v_1,v_2,v_3\}$, and these are not vertices of $\Gamma_2$, we have that $\Gamma_1$ and $\Gamma_2$ intersect in their boundary. Moreover, if we apply $\varphi$ to all the $3$-elements subsets which are contained in exactly one element in $F$ we obtain precisely the faces in \eqref{eq: boundary}. Thus, $\partial\Gamma_1=\partial\Gamma_2$. Finally, let
\[
    \Delta_{12,33}:=\Gamma_1\cup \Gamma_2.
\]
The notation indicates that $\Delta_{12,33}$ is a simplicial complex on $12$ vertices and $33$ facets ($24$ coming from $\Gamma_1$ and $9$ from $\Gamma_2$). These facets are listed in \eqref{eq: facets delta}. Via the so-called Dehn-Sommerville relations the number of edges and $2$-faces of $\Delta_{12,33}$ can be deduced. Its $f$-vector (the vector recording the number of faces in each dimension) equals $f(\Delta_{12,33})=(1,12,45,66,33)$.
\begin{align}\label{eq: facets delta}
   \nonumber& x_1y_2y_3z_1, y_2y_3z_1v_2, y_1y_2y_3v_2, y_1y_2z_3v_2, y_2z_1z_3v_2, x_3y_2z_1z_3, x_1x_3y_2z_1,\\
\nonumber&x_1x_3z_1v_3, x_3y_1y_2z_3, z_1z_2z_3v_2, x_3y_1y_2v_1, x_2x_3y_1v_1, x_3z_1z_3v_3, z_1z_2z_3v_3,\\
&x_1z_1z_2v_3, x_1y_3z_1z_2, y_3z_1z_2v_2, y_1y_3v_2z_2, x_1x_3y_2v_1, x_1x_2x_3v_1, x_1y_2y_3v_1,\\
\nonumber&x_2x_3y_1z_3, x_2x_3z_3v_3, y_1z_2z_3v_2, x_2y_1z_2z_3, x_2y_1z_2y_3, y_1y_2y_3v_1, x_2y_1y_3v_1,\\
\nonumber&x_1x_2y_3v_1, x_1x_2y_3z_2, x_1x_2x_3v_3, x_1x_2z_2v_3, x_2z_2z_3v_3.
\end{align}
Before stating the main result, we highlight some features of the simplicial complex $\Delta_{12,33}$ in the following remarks.
\begin{remark}\label{rem: tau}
     $\Delta_{12,33}$ has a rich automorphism group. In particular the permutation
    \[
        \tau = (x_3,y_2,z_1,x_1,y_3,z_2,x_2,y_1,z_3)(v_2,v_3,v_1)
    \]
    is an automorphism of $\Delta_{12,33}$. Thus, the links of vertices in $\{v_1,v_2,v_3\}$ are isomorphic, and so are the links of vertices in $V\setminus\{v_1,v_2,v_3\}$. In particular, $\lk_{\Delta_{12,33}}(v_1)\cong \Diamond_3$, while $\lk_{\Delta_{12,33}}(x_1)$ is isomorphic to the flag $2$-sphere on $8$ vertices depicted in \Cref{fig: link}.
\end{remark}
\begin{remark}\label{rem: on delta}
    By direct inspection on its graph, we make the following observation on $\lk_{\Delta_{12,33}}(x_1)$. Let $E=\lk_{\Delta_{12,33}}(x_1)[S]$ be an equator (an induced cycle) of $\lk_{\Delta_{12,33}}(x_1)$. If $S\cap \{v_1,v_2,v_3\}=\emptyset$, then either $S=\{x_2,x_3,y_2,y_3\}$ or $S=\{x_2,x_3,y_3,z_1\}$. In particular, $\{x_2,x_3\}\subset S$ and $E$ is a $4$-cycle. We will make use of these two facts in the proof of \Cref{thm: equators}. Observe that, acting on $\Delta_{12,33}$ with the automorphism $\tau$ in \Cref{rem: tau} we obtain that any induced cycle in $\lk_{\Delta_{12,33}}(x_i)$ (respectively $\lk_{\Delta_{12,33}}(y_i)$, $\lk_{\Delta_{12,33}}(z_i)$) not containing vertices in $\{v_1,v_2,v_3\}$ is a $4$-cycle, and it contains $x_j$ (resp. $y_j$, $z_j$) for every $j\in\{1,2,3\}\setminus\{i\}$.
\end{remark}
\begin{remark}
    Observe that neither $\Gamma_1$ nor $\Gamma_2$ are flag simplicial complexes. As all internal vertices of $\Gamma_1$ are pairwise non-adjacent and the only faces in the interior which do not contain a vertex in $\{v_1,v_2,v_3\}$ are $x_1x_2x_3$, $y_1y_2y_3$ and $z_1z_2z_3$, the minimal non-faces of $\Gamma_1$ are precisely the minimal non-faces of $\partial\Gamma_1$ which are not $x_1x_2x_3$, $y_1y_2y_3$ and $z_1z_2z_3$. The missing triangles are the $9$ subsets of the form $\{x_i,y_j,z_k\}$ in which vertices are pairwise adjacent:
    \begin{equation}\label{eq: misiing gamma1}
        \{
        x_1y_2z_1, x_1y_3z_1, x_1y_3z_2, x_2y_1z_2, x_2y_1z_3, x_2y_3z_2,  x_3y_1z_3, x_3y_2z_1, x_3y_2z_3\}.  
    \end{equation}
    $\Gamma_2$ doesn't have any interior face of dimension smaller than $2$. Therefore, its minimal non-faces are contained in those of $\partial\Gamma_2$
\end{remark}
\begin{figure}[h]
    \centering
    \includegraphics[scale=0.5]{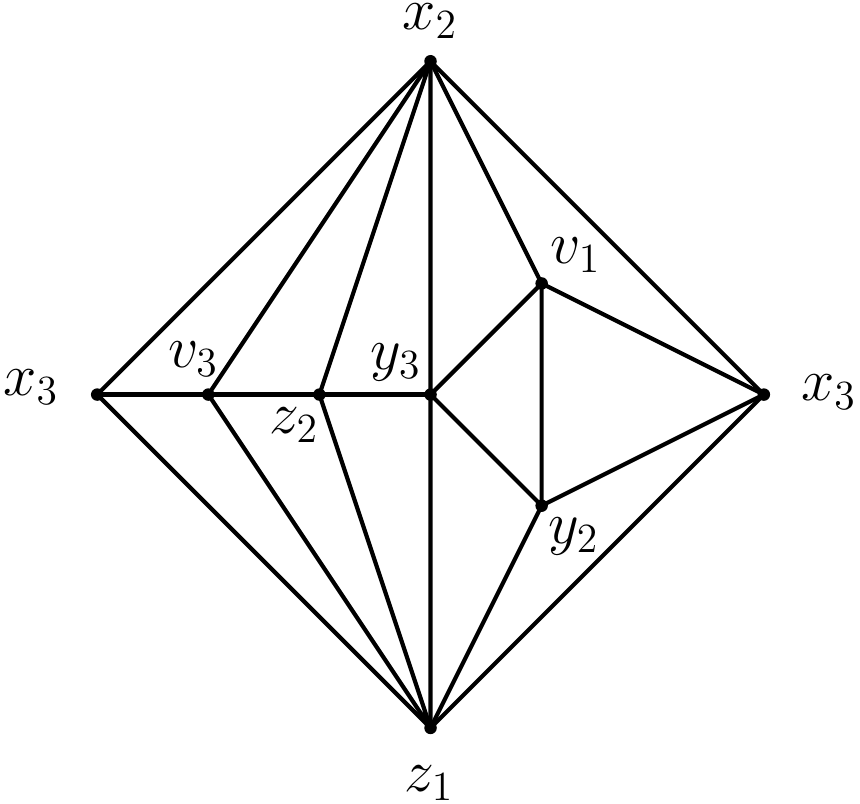}
    \qquad\includegraphics[scale=0.3]{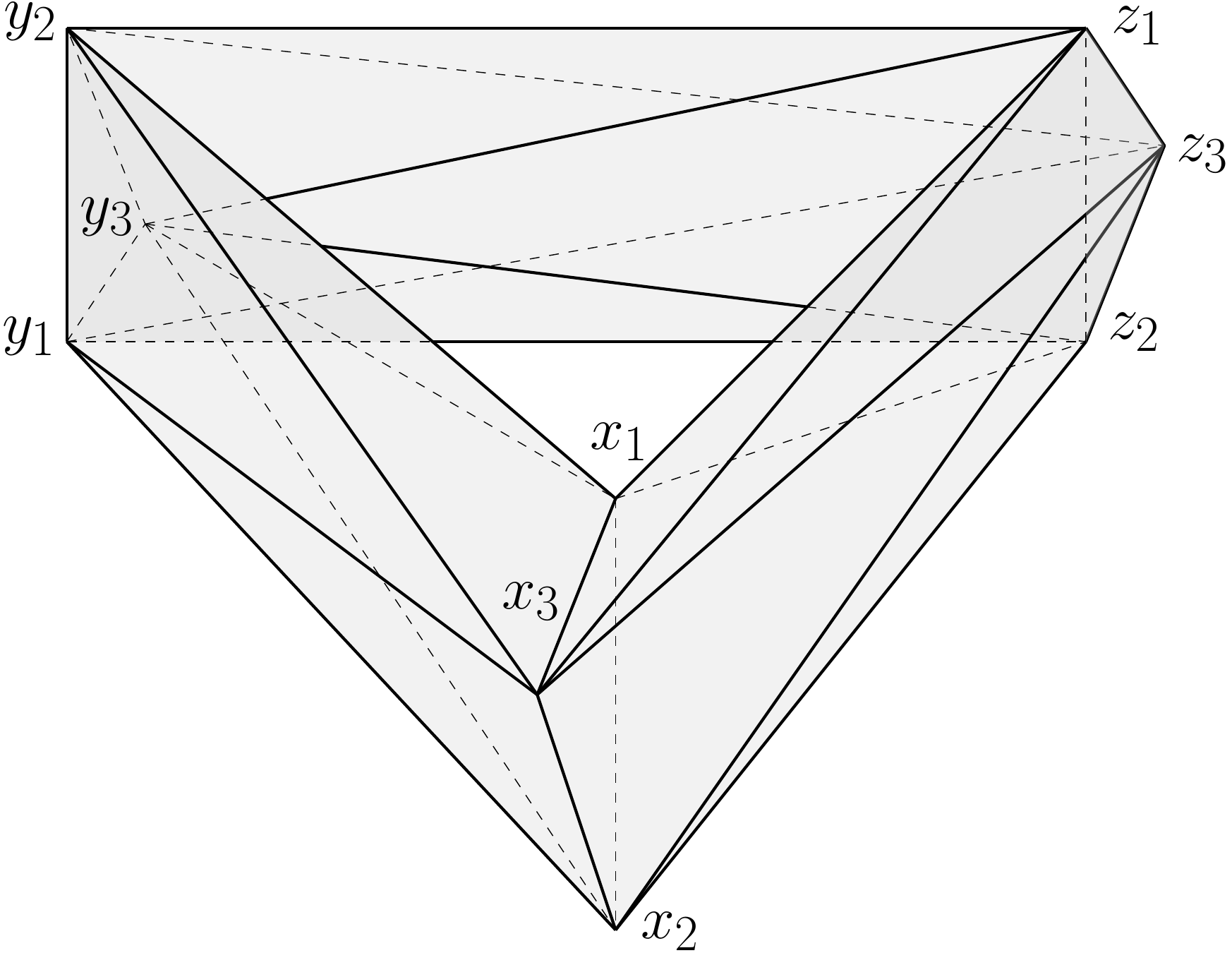}
    \caption{The $2$-sphere $\lk_{\Delta_{12,33}}(x_1)$ (left) and the simplicial complex $\partial\Gamma_1=\partial\Gamma_2$ (right).}
    \label{fig: link}
\end{figure}
\begin{lemma}\label{lem: equators psi}
    Let $M$ be a simplicial complex with $|M|\cong \mathbb{D}^2\times \mathbb{S}^1$. Assume $M$ does not have vertices in its interior. Then there is no induced equator in $M$.
\end{lemma}
\begin{proof}
    $M$ is a \emph{$2$-stacked} simplicial manifold with boundary (see e.g. \cite{Ba}). In \cite[Corollary 4, Theorem 7]{Ba} the author proves that for every induced subcomplex $E$ of $M$ the map $H_2(E;K)\to H_2(M;K)$ induced by the inclusion is injective. As $H_2(M;K)=0$ the claim follows. 
\end{proof}
The simplicial complex $\Gamma_2$ satisfies the conditions of \Cref{lem: equators psi}. If $S\subseteq V(\Delta_{12,33})$, is a subset which does not contain any vertex $v_{i}$ and such that $\{x_{1},x_{2},x_{3}\}\nsubseteq S$, $\{y_{1},y_{2},y_{3}\}\nsubseteq S$ and $\{z_{1},z_{2},z_{3}\}\nsubseteq S$, then $\Delta_{12,33}[S]=\Gamma_2[S]$. In particular, \Cref{lem: equators psi} shows that $\Delta_{12,33}[S]$ is not an equator. 
\begin{theorem}\label{thm: 3-sphere}
$\Delta_{12,33}$ is a flag simplicial $3$-sphere.
\end{theorem}
\begin{proof}
    To show that among the lens spaces $|\Delta_{12,33}|$ is a $3$-sphere we employ the lower bound theorem for simplicial manifolds \cite[Theorem 1.6]{Mur}. If for any field $K$ a simplicial $3$-manifold $\Delta$ satisfies  $\dim_{K}H_1(\Delta; K)>0$, then $e\geq 4n$, with $n$ and $e$ the number of vertices and edges respectively. Here $H_1(\Delta; K)$ denotes the first simplicial homology group of $\Delta$ with coefficients in $K$. As for $\Delta_{12,33}$ we have $e=45<48=4n$, it implies that $H_1(\Delta; K)=0$ over any field. We conclude, as the $3$-sphere is the only lens space satisfying this property. 
    In order to prove that $\Delta_{12,33}$ is flag we have to verify that all minimal non-faces of cardinality $3$ of $\Gamma_1$, which are listed in \eqref{eq: misiing gamma1} are faces of $\Gamma_2$. This holds, applying $\varphi^{-1}$ to the sets in \eqref{eq: misiing gamma1} we obtain the $9$ underlined subsets of the facets in \eqref{eq: facets abc}.
\end{proof}
Another way to see that $\Delta_{12,33}$ is indeed a simplicial $3$-sphere is to verify that the ordering of the facets of $\Delta_{12,33}$ in \eqref{eq: facets delta} is a \emph{shelling order} (see \cite[III.2]{StaGreen}, and it is well known that any shellable simplicial manifold is a simplicial sphere.\\
We now turn our attention to the main properties of $\Delta_{12,33}$ related to \Cref{qu: question}.
\begin{theorem}\label{thm no contractible}
$\Delta_{12,33}$ is not a suspension and it has no contractible edge.
\end{theorem}
\begin{proof}
    The vertices in the graph of $\Delta_{12,33}$ have either degree $8$ or $6$. This implies that every vertex is non-adjacent to at least two vertices, and hence $\Delta_{12,33}$ is not a suspension.\\
    We now show that for every edge $e\in\Delta_{12,33}$ there is an induced $4$-cycle containing $e$. We consider separately edges which do not contain any vertex in $\{v_1,v_2,v_3\}$ (Case 1) and those which contain exactly one  of these vertices (Case 2). Observe that there are no edges of the form $\{v_i,v_j\}$.
    \begin{itemize}
        \item[Case 1.] Every edge of the form $\{x_i,y_j\}$, $\{y_i,z_j\}$ and $\{x_i,z_j\}$ is contained in an induced subcomplex of $\Delta_{12,33}$ which is isomorphic to the boundary complex of a $3$-dimensional cross-polytope. In particular, every edge of this form is contained in an induced $4$-cycle.
        \item[Case 2.] Observe that $\{v_1,x_1,z_2,y_1\}$, $\{v_1,x_2,z_3,y_2\}$ and $\{v_1,x_3,z_1,y_3\}$ are induced $4$-cycles of $\Delta_{12,33}$. In particular every edge of the form $\{v_1,x_i\}$, $\{v_1,y_i\}$ or $\{v_1,z_i\}$ is contained in an induced $4$-cycle. Similarly, $\{v_2,y_1,x_3,z_1\}$, $\{v_2,y_2,x_1,z_2\}$ and $\{v_2,y_3,x_2,z_3\}$ are induced $4$-cycle, which shows that every edge of the form $\{v_2,x_i\}$, $\{v_2,y_i\}$ or $\{v_2,z_i\}$ is contained in an induced $4$-cycle. Finally, the $4$-cycles $\{v_3,x_1,y_2,z_3\}$, $\{v_2,x_2,y_3,z_1\}$ and $\{v_2,x_3,y_1,z_2\}$ ensure that also the edges of the form $\{v_3,x_i\}$, $\{v_3,y_i\}$ or $\{v_3,z_i\}$ are not contractible.
    \end{itemize}
    \end{proof}
    The following result shows that $\Delta_{12,33}$ answers \Cref{qu: question} in dimension $d=3$.
    
    \begin{theorem}\label{thm: equators}
     All induced equators of $\Delta_{12,33}$ are vertex links.
    \end{theorem}
    \begin{proof}
        Let $E=\Delta_{12,33}[S]$ be an induced equator of $\Delta_{12,33}$. We treat the cases $S\cap \{v_1,v_2,v_3\}\neq \emptyset$ and $S\cap \{v_1,v_2,v_3\}=\emptyset$ separately.
    \begin{itemize}
        \item[-] $S\cap\{v_1,v_2,v_3\}\neq \emptyset$. We can assume w.l.o.g. that $v_1\in S$. As observed in \Cref{rem: tau}, $\lk_{\Delta_{12,33}}(v_i)$ is isomorphic to $\Diamond_3$. Moreover, $\lk_{E}(v_1)$ must be an induced cycle of $\lk_{\Delta_{12,33}}(v_1)$, i.e., a $4$-cycle. Assume this $4$-cycle is $\{x_1,y_1,y_2,x_2\}$. This assumption is not restrictive, since applying $\tau$ is transitive on the $4$-cycles of $\lk_{\Delta_{12,33}}(v_1)$, i.e., $\tau^3$ and $\tau^6$ fix $v_1$ but permute the $4$-cycles in its link. Observe that if $x_3\in S$ or $y_3\in S$ then $\{v_1,x_1,x_2,x_3\}\subset S$ or $\{v_1,y_1,y_2,y_3\}\subset S$. As these two sets support $3$-simplices of $\Delta_{12,33}$ this would contradict the fact that $E$ is $2$-dimensional. We then have that $x_3\notin S$ and $y_3\notin S$.\\
    If we assume that $E$ is not the link of any vertex of $\Delta_{12,33}$, then either $\{z_3,v_2\}\subset S$ or $\{z_2,v_3\}\subset S$. This holds since $\lk_{\Delta_{12,33}}(x_3)=\Delta_{12,33}[V\setminus\{y_3,z_2,v_2\}]$ and
    $\lk_{\Delta_{12,33}}(y_3)=\Delta_{12,33}[V\setminus\{x_3,z_3,v_3\}]$, and hence $S$ cannot be contained (and cannot contain) $V\setminus\{y_3,z_2,v_2\}$ or $V\setminus\{x_3,z_3,v_3\}$. We then distinguish between two subcases.
    \begin{itemize}
        \item[-] $\{z_3,v_2\}\subset S$. From what discussed above, this would imply that one of the induced $4$-cycles of $\lk_{\Delta_{12,33}}(v_2)$ belongs to $E$. There is only one such cycle not containing $y_3$, namely $\{y_1,y_2,z_1,z_2\}$. This implies that $\{v_2,z_1,z_2,z_3\}\subset S$. As $\Delta_{12,33}[v_2,z_1,z_2,z_3]$ is a $3$-simplex, this contradicts the fact that $E$ is $2$-dimensional. 
        \item[-] $\{z_2,v_3\}\subset S$. The only induced $4$-cycle of $\lk_{\Delta_{12,33}}(v_3)$ which does not contain $x_3$ is $\{x_1,x_2,z_3,z_2\}$. Then $\{v_3,z_1,z_2,z_3\}\subset E$. Similarly to the previous case, $\Delta_{12,33}[v_3,z_1,z_2,z_3]$ is a $3$-simplex which contradicts the fact that $E$ is $2$-dimensional. 
    \end{itemize}
    \item[-] $S\cap\{v_1,v_2,v_3\}= \emptyset$. W.l.o.g. assume that $x_1\in S$. As $\lk_{\Delta_{12,33}}(x_1)$ is isomorphic to the $2$-sphere in \Cref{fig: link}, the link of $x_1$ in $E$ is then an induced cycle of $\lk_{\Delta_{12,33}}(x_1)$ and by \Cref{rem: on delta} such cycle must be a $4$-cycle. Therefore, $E$ is a flag $2$-sphere in which all vertex links are $4$-cycles. By \cite[Lemma 5.3]{NP} $E$ is isomorphic to $\Diamond_3$. Moreover, again by \Cref{rem: on delta}, either $\{x_2,x_3,y_2,y_3\}\subset S$ or $\{x_2,x_3,y_3,z_1\}\subset S$. In the first case $S=\{x_1,x_2,x_3,y_1,y_2,y_3\}$, which implies $\Delta_{12,33}[S]=\lk_{\Delta_{12,33}}(v_1)$. In the second case \Cref{rem: on delta} implies that $S=\{x_1,x_2,x_3,y_1,y_2,y_3,z_1,z_2,z_3\}$, which contradicts the fact that $E\cong \Diamond_3$.
    \end{itemize}
\end{proof}
\begin{remark}\label{rem: delta is minimal}
    $\Delta_{12,33}$ has the minimum number of vertices among all the complexes satisfying the conditions in \Cref{qu: question}. This follows from \cite{NP}, where the authors characterize flag homology spheres (a generalization of simplicial spheres) with few vertices. In particular, it follows from \cite[Proposition 5.4, 5.5, 5.6] {NP} that any flag simplicial $3$-sphere $\Delta$ with $f_0(\Delta)\leq 11$ has a contractible edge.
\end{remark}
\section{Simplicial join}
In this last section we show that the properties proved in \Cref{thm no contractible} and \Cref{thm: equators} are preserved taking simplicial join.
\begin{proposition}\label{pro: join}
    Let $\Phi_1$ be a flag simplicial $m$-sphere and $\Phi_2$ be a flag simplicial $n$-sphere such that \begin{itemize}
        \item[(i)] $\Phi_1$ and $\Phi_2$ have no contractible edge and they are not a suspension;
        \item[(ii)] All the induced equators of $\Phi_1$ and $\Phi_2$ are links of a vertex
    \end{itemize}
    Then $\Phi_1*\Phi_2$ is a flag simplicial $(n+m+1)$-sphere satisfying properties (i) and (ii).
\end{proposition}
\begin{proof}
    It is well known that $\Phi_1*\Phi_2$ is a flag simplicial sphere, and that the $\Phi_1*\Phi_2$ is a suspension if and only if either $\Phi_1$ or $\Phi_2$ is a suspension.\\
    By definition, the edges of $\Phi_1*\Phi_2$ are either edges of $\Phi_1$, edges of $\Phi_2$ or edges $\{a_1,a_2\}$, with $a_i\in\Phi_i$. Edges of the first two types are contained induced $4$-cycles of $\Phi_1$ or $\Phi_2$, which are induced $4$-cycles of $\Phi_1*\Phi_2$. In the third case, $\{a_1,a_2,b_1,b_2\}$ is an induced $4$-cycle, for any $b_1\in\Phi_1$ not adjacent to $a_1$ and any $b_2\in\Phi_2$ not adjacent to $a_2$. Observe that such $b_1$ and $b_2$ exist, as any vertex in any flag sphere which is not a suspension is not adjacent to at least two vertices.\\
    Let $E$ be an induced equator of $\Phi_1*\Phi_2$. Then $E=\Sigma_1*\Sigma_2$, where $\Sigma_1$ is an induced subcomplex of $\Phi_1$ homeomorphic to a $(k-1)$-sphere, and $\Sigma_2$ is an induced subcomplex of $\Phi_2$ homeomorphic to a $(n+m-k)$-sphere. By dimension reasons, we have $k=n$ or $k=n+1$. In the first case $E=\Sigma_1*\Phi_2=\lk_{\Phi_1}(v)*\Phi_2=\lk_{\Phi_1*\Phi_2}(v)$, while in the second case $E=\Phi_1*\Sigma_2=\Phi_1*\lk_{\Phi_2}(v)=\lk_{\Phi_1*\Phi_2}(v)$. 
\end{proof}
Combining this result with the construction in \Cref{sec 2} we obtain the following.
\begin{corollary}
    For every $k\geq 0$ the $k$-fold join $\Delta_{12,33}^k:=\underbrace{\Delta_{12,33}*\cdots*\Delta_{12,33}}_{k\text{-times}}$ is a flag simplicial $(4k+3)$-sphere such that:
    \begin{itemize}
        \item[-] $\Delta_{12,33}^k$ has no contractible edge and it is not a suspension;
        \item[-] All the induced equators of $\Delta_{12,33}^k$ are links of a vertex.
    \end{itemize}
\end{corollary}
\section{Conclusion}
We are currently studying how to generalize our construction to obtain flag simplicial $d$-spheres answering \Cref{qu: question} in the positive for any $d$ or with more vertices. Even though we proved in \Cref{pro: join} that taking joins gives one way to produce such complexes, it would be considerably more interesting to find flag spheres in any dimension satisfying the conditions in \Cref{qu: question} which are not the join of spheres of smaller dimension. Following our construction, at least in the $3$-dimensional case, would require finding suitable solid tori which glue together to a flag sphere satisfying the properties in \Cref{qu: question}. The most challenging task in generalizing our result appears to be the description of induced equators. All the same, we believe there exist many such simplicial complexes and investigating their structure would be relevant to tackle the conjectures posed in \cite{Gal} and \cite{NP}. We ran experiments performing a sequence of random edge subdivisions followed by random edge contractions on a flag sphere, as by \cite{LuN} these operations suffice to connect any two flag PL spheres. In this way we found more examples of $3$-spheres giving a positive answer to \Cref{qu: question}, on 13, 14, 15 and 16 vertices. We make the list of facets of these simplicial complexes as well as the code used for the experiments available at \cite{VenGit}. We conclude with the following question.
\begin{question}
    Fix $d\geq 3$. Are there infinitely many flag $d$-spheres which satisfy conditions (i), (ii) and (iii) in \Cref{qu: question}?
\end{question}
\section*{Acknowledgement}
We thank Alessio D'Al\`{i} and Joseph Doolittle and Martina Juhnke-Kubitzke for helpful discussions and Eran Nevo for suggesting the problem and for pointing out the observation in \Cref{rem: delta is minimal}. 

\bibliography{bibliography}
\bibliographystyle{alpha}
\end{document}